\documentclass[letterpaper, 10 pt, conference, twocolumn]{ieeeconf} 
\IEEEoverridecommandlockouts
 \pdfoutput=1
\usepackage[utf8]{inputenc}
\usepackage[T1]{fontenc}
\usepackage{url}
\usepackage{amsmath}
\usepackage{amsfonts}
\usepackage{amssymb}
\usepackage{algorithm}
\usepackage{algpseudocode}
\usepackage{bbm}
\usepackage{graphicx}
\usepackage{caption}
\usepackage{subcaption}
\usepackage[usenames, dvipsnames]{color}
\usepackage{xcolor}
\usepackage{cite}
\usepackage{ifthen}
\usepackage{mathtools}
\usepackage{comment}
\newtheorem{theorem}{Theorem}
\newtheorem{corollary}{Corollary}

\newtheorem{lemma}{Lemma}
\newtheorem{remark}{Remark}
\newtheorem{proposition}{Proposition}
\newtheorem{definition}{Definition} 
\newtheorem{example}{Example}
\usepackage{flushend} 
\usepackage{pgfplots}
\pgfplotsset{compat=1.18}

\newboolean{showcomments}
\setboolean{showcomments}{false}

\newcommand{\david}[1]{\ifthenelse{\boolean{showcomments}}
{{(David says: #1)}}{}}
\newcommand{\emma}[1]{\ifthenelse{\boolean{showcomments}}
{\textcolor{VioletRed}{(Emma says: #1)}}{}}
\newcommand{\edit}[1]{\ifthenelse{\boolean{showcomments}}
{{#1}}{}}

\title{\LARGE \bf Positive Observers Revisited}

\author{ {David Ohlin, Anders Rantzer and Emma Tegling}
\thanks{The authors are with the Department of Automatic Control and ELLIIT Strategic Research Area at Lund University, Lund, Sweden. Email: \{{\tt\small{david.ohlin, anders.rantzer, emma.tegling}\}@control.lth.se}}\thanks{This work is partially funded by Wallenberg AI, Autonomous Systems and Software Program (WASP) and the European Research Council (ERC) under the European Union’s Horizon 2020 research and innovation programme under grant agreement No 834142 (ScalableControl).}}

\begin{document}
\maketitle

\begin{abstract}
    The paper shows that positive linear systems can be stabilized using positive Luenberger-type observers. This is achieved by structuring the observer as monotonically converging upper and lower bounds on the state. Analysis of the closed-loop properties under linear observer feedback gives conditions that cover a larger class than previous observer designs. The results are applied to nonpositive systems by enforcing positivity of the dynamics using feedback from the upper bound observer. The setting is expanded to include stochastic noise, giving conditions for convergence in expectation using feedback from positive observers.
\end{abstract}

\section{Introduction}

In this paper, we revisit the basic setting of observer synthesis for positive linear systems. Monotonicity of the positive dynamics naturally admits a structuring of two observers as upper and lower bounds on the state. Within the literature on interval observers the typical approach uses time-varying coordinate transforms \cite{Khan21survey,zhang23survey,ngoc25interval} to obtain positive error dynamics. This then allows observers to be invariantly ordered in relation to the state. In the case of inherently positive linear systems, this can be done without coordinate transformations. Considering the problems of estimation and control jointly, we instead take a direct approach to the design of bounds and feedback, giving linear conditions that ensure stability and positivity of the error dynamics. This resolves the apparent impossibility, as concluded in \cite{aitrami06observation}, of stabilizing feedback from positive observers. The conditions generate valid observers for a larger set of systems than those previously covered in \cite{aitrami06observation} as a result of treating the full, coupled problem. We go on to show that these results are valid beyond inherently positive systems, by using observer feedback to enforce positivity of the closed-loop dynamics. In the more general setting of noisy dynamics, we derive a second set of linear constraints to guarantee stabilization in the presence of stochastic disturbances.

Positive systems are frequent and varied in the study of the natural world. When only certain aspects of the state can be measured but positivity of the system is inherent, it is natural to require positivity of an observer that seeks to reconstruct the full state. In the linear setting, positivity of an estimate $\hat{x}$ is required to guarantee that linear feedback $u = K\hat{x}$ constitutes a valid input. The problem of positive stabilization in continuous time is treated in \cite{deleenheer01stabilization} for single-input systems with certain assumptions on the dynamics. The later work \cite{aitrami07control} solves the problem for a larger class of discrete-time systems, showing how feasible feedback gains can be synthesized via linear programming. Both of these works assume full knowledge of the state. Subsequently, \cite{briat11ILC} gives linear programming conditions for the synthesis of positivity preserving output-feedback controllers within specified gain bounds, without introducing observers.

The observation problem for linear positive systems is initially treated in \cite{hof98compartmental}, where it is applied to compartmental systems used for biological modeling. In \cite{aitrami06observation}, the authors give a linear program for the synthesis of observer gains that ensure positivity and asymptotic convergence to the true state. A negative result regarding the positivity of closed-loop error dynamics under stabilizing linear feedback leads the authors to conclude that stabilization is not possible. In this paper we show that a minor variation of the underlying assumptions allows for the construction of stabilizing feedback from positive linear state estimates on standard Luenberger form. 

On the other hand, we may desire the system dynamics to behave like a positive system, even when this is not inherent. For example, the recent work \cite{ito25positivizing} develops dissipativity theory to impose positivity around fixed points other than the origin in nonlinear systems of positive quantities. Although the output-based approach taken in the present work differs from~\cite{ito25positivizing}, we adopt the terminology of "positivization" to refer to the enforcing of positive dynamics through feedback. For linear systems with full state information, this problem is treated in~\cite{aitrami05positivization}. This can be done for a multitude of reasons. In the case of physical dynamics, positivization is useful to ensure that states only take values above a certain setpoint. Positivity can also be imposed to leverage the theory of positive systems (see \cite{rantzervalcher18tutorial} for a review of basic properties) when synthesizing controllers and observers, facilitating e.g. scalable distributed calculation of optimal controllers \cite{ohlin25heuristic}. 

The following section sets up the basic premise of discrete-time positive linear systems, together with definitions of used notation. In Section III we introduce the structure of the observers and present the main results of the paper, which give conditions for stabilization using feedback from linear positive state estimates. Section IV discusses the implications of the derived constraints compared to previous results and showcases these in two examples. In Section V, a more realistic setting with state and measurement noise affecting the system is treated. Finally, Section VI summarizes the conclusions and sets out possible directions for future work.

\section{Preliminaries}

\subsection{Notation}
Inequalities are applied element-wise for matrices and vectors throughout. Further, the notation $\mathbb{R}^n_{+}$ is used to denote the closed nonnegative orthant of dimension $n$. The expressions $\mathbf{1}_{p\times q}$ and $\mathbf{0}_{p\times q}$ signify a matrix of ones or zeros, respectively, of the indicated dimension, with subscript omitted when the size is clear from the context. If the dimension is zero, this is to be interpreted as the empty matrix. Let $\rho(A)$ denote the spectral radius of the matrix $A$, equal to the magnitude of the largest eigenvalue of $A$. The matrix $A$ is said to be Schur stable (equivalently, a Schur matrix) if $\rho(A)<1$. When convenient we use the shortened notation $x^+=f(x)$ instead of $x(t+1) = f(x(t))$, dropping the time argument, to indicate the propagation of dynamics. 

\subsection{Problem setup}

This paper concerns the estimation problem for systems
\begin{equation}\label{eq:sys}
    \begin{aligned}
        x(t+1) &= Ax(t) + Bu(t)\\
        y(t) &= Cx(t).
    \end{aligned}
\end{equation}
We restrict ourselves to the class of positive systems, meaning that the nonnegative orthant $x\in\mathbb{R}^n_+$ is invariant under the dynamics \eqref{eq:sys} for all inputs in some set $u(t)\in\mathcal{U}$. For this to hold, we require that $A\ge0$. A variety of alternative definitions are used in the literature, e.g. \cite{rantzervalcher18tutorial} requiring positivity of the state for all nonnegative inputs $u(t)\ge0$. This somewhat restrictive assumption implies that $B\ge0$, preventing the input from stabilizing the system \cite{deleenheer01stabilization}. Defining the set of allowed inputs $\mathcal{U}$ as a compact subset of $\mathbb{R}^n_+$ instead (as in \cite{aitrami06observation}, \cite{ohlin25heuristic}) permits stabilizing feedback that maintains positivity of the closed-loop system. Note that we place no restriction on the sign of the output $y\in\mathbb{R}^p$. Let $\phi(\xi,t)$ denote the trajectory of a system at time~$t$ starting from $x(0) = \xi$. The following is an elementary property of positive linear systems \cite{angeli03monotone}.
\begin{proposition}{\textit{(Monotonicity)}}\label{prop:mono}
    Consider the system ${x(t+1) = Ax(t)}$ with ${A\ge0}$. For ${a, b\in\mathbb{R}^n_+}$ it holds that 
    \begin{equation*}
        a\ge b \implies \phi(a,t)\ge\phi(b,t) \;\;\textnormal{for all}\;\;t.
    \end{equation*}
\end{proposition}

\vspace{2mm}

\section{Positive observers}

The state is estimated using an observer on the standard Luenberger form
\begin{equation}\label{eq:obs}
    \hat{x}(t+1) = (A-LC)\hat{x}(t) + Ly(t) + Bu(t).
\end{equation}

Despite the feasibility of a positive observer, the following negative result of~\cite{aitrami06observation} shows that it is not possible to use an estimate for which $\mathbb{R}^n_+$ is invariant under~\eqref{eq:obs} to stabilize the system~\eqref{eq:sys}.
\begin{proposition}{(\cite[Lemma 4.1]{aitrami06observation})}\label{prop:aitrami}
    Assume that $A$ is not Schur. Then, there does not exist a positive observer, i.e. an observer for which ${\hat{x}\in\mathbb{R}^n_+}$ is invariant, on the form \eqref{eq:obs} for system \eqref{eq:sys}, such that observer feedback control on the form $u(t) = K\hat{x}$ is asymptotically stabilizing. 
\end{proposition}

This has discouraged the study of positive observers, as they are seemingly of little practical use. Invariance of the full orthant $\mathbb{R}^n_+$ is, however, not a necessity for linear positive estimation. First observe that, in the noise-free system \eqref{eq:sys}, monotonicity (Proposition~\ref{prop:mono}) implies that any element of the estimate initialized above or below the corresponding element of the state remains so under the update \eqref{eq:obs}. We return later to the case when dynamics and measurements are corrupted by stochastic noise. Within the field of interval observers \cite{Khan21survey}, \cite{zhang23survey}, the synthesis of positive estimation dynamics is for this reason used to provide bounds on the state. This property motivates the separate study of upper and lower bound estimates for positive systems:
\begin{definition}{\textit{(Upper and lower estimates)}}\label{def:estimates}
    Denote by $\overline{x}$ and $\underline{x}$ the \textit{upper} and \textit{lower state estimate} respectively, evolving according to the observer dynamics \eqref{eq:obs} with corresponding gains $\overline{L}$ and $\underline{L}$. 
\end{definition}
With estimates explicitly separated, the requirement on invariance in Proposition \ref{prop:aitrami} can be modified. Let the subspace ${\mathcal{X}\subset\mathbb{R}^{3n}_+}$ be defined as
\begin{equation}
    \mathcal{X} = \{x,\underline{x},\overline{x}\in\mathbb{R}^n_+ : \underline{x}\le x\le\overline{x}\}.
\end{equation}
As will be shown below, the strengthening to require invariance of this subspace in turn relaxes the conditions on the closed-loop operator. This remedies the flaw pointed out in~\cite{aitrami06observation} and allows for stabilization using positive estimates. The tradeoff in terms of information is that initialization of the upper estimate in $\mathcal{X}$ requires prior knowledge of an upper bound on the initial state.

\subsection{Closing the loop}

Adopting a linear state feedback from both estimates, consider the following structure
\begin{equation}\label{eq:u}
    u(t) = \underline{K}\;\underline{x}(t) + \overline{K}\;\overline{x}(t).
\end{equation}
Here, a positive feedback convention is chosen to maintain consistency with the literature on positive systems. Separating the feedback gains $\overline{K}$ and $\underline{K}$ is common practice in the context of interval observers, see \cite{Khan21survey}, \cite{zhang23survey}. In the present work, this structure will be necessary in order to utilize the different properties of the upper and lower bounds. Closing the loop with feedback on the form \eqref{eq:u} yields the extended system
\begin{equation}\label{eq:fullx}
    \!\!\!\!\!\!\begin{bmatrix}
        x \\ \overline{x} \\ \underline{x}
    \end{bmatrix}^+ \!\! = \!\!
    \begin{bmatrix}
        A & B\overline{K} & B\underline{K} \\ \overline{L}C & A-\overline{L}C+B\overline{K} & B\underline{K} \\  \underline{L}C & B\overline{K} & A-\underline{L}C+B\underline{K}
    \end{bmatrix}\!\!\!\!\begin{bmatrix}
        x \\ \overline{x} \\ \underline{x}
    \end{bmatrix}\!\!.\!
\end{equation}
The following lemma gives necessary and sufficient conditions for invariance of $\mathcal{X}$ in the closed-loop system~\eqref{eq:fullx}. 
\begin{lemma}\label{lemma:1}
    Consider the system \eqref{eq:fullx}. The subspace $\mathcal{X}$ is invariant if and only if the following conditions hold:
    \begin{subequations}\label{eq:conditions}
        \begin{align}
        A+B(\overline{K}+\underline{K})&\ge0\label{subeq:ABKK}\\ 
        A+B\overline{K}&\ge0\label{subeq:ABK}\\ B\overline{K}&\ge0\label{subeq:BK}\\ A-\overline{L}C&\ge0\label{subeq:ALCu}\\ A-\underline{L}C&\ge0\label{subeq:ALCl} \\ B\overline{K}+\underline{L}C&\ge0\label{subeq:BKLC}.
        \end{align}
    \end{subequations}
\end{lemma}
\begin{proof}
    Define the errors $\underline{e} = x-\underline{x}$ and $\overline{e} = \overline{x}-x$. Invariance of $\mathcal{X}$ for the state and estimates is equivalent to invariance of the cone
\begin{equation}
    \mathcal{X}_e = \{x,\underline{e},\overline{e}\in\mathbb{R}^n_+ : \underline{e}\le x\}.
\end{equation}
under the extended error dynamics
\begin{equation}\label{eq:errors}
    \!\!\!\!\!\begin{bmatrix}
        x \\ \overline{e} \\ \underline{e}
    \end{bmatrix}^+ \!\!=\!\! 
    \begin{bmatrix}
        A+B(\overline{K}+\underline{K}) & B\overline{K} & -B\underline{K} \\ 0 & A-\overline{L}C & 0 \\ 0 & 0 & A-\underline{L}C
    \end{bmatrix}\!\!\!\!\begin{bmatrix}
        x \\ \overline{e} \\ \underline{e}
    \end{bmatrix}\!\!.\!
\end{equation}

(${i\textnormal{)}\implies\textnormal{(}ii}$): In~\eqref{eq:errors}, it is apparent that positivity of the diagonal elements $A-\overline{L}C$ and $A-\underline{L}C$ is necessary and sufficient for positivity of the errors. Inspecting the first row of \eqref{eq:errors}, we see that the stated conditions are sufficient to ensure ${x^+\ge0}$ for all ${(x,\overline{e},\underline{e})\in\mathcal{X}_e}$. For sufficiency, it remains to show that the conditions imply ${x^+\ge\underline{e}^+}$ {(or equivalently, ${\underline{x}^+\ge0}$)}, which yields the inequality
\begin{equation}\label{eq:lowerpos}
    B\overline{K}\overline{x}+\underline{L}Cx+(A+B\underline{K}-\underline{L}C)\underline{x}\ge0.
\end{equation}
The ordering ${\underline{x}\le x\le\overline{x}}$ together with conditions \eqref{subeq:BK} and~\eqref{subeq:BKLC}, respectively, imply the following successive lower bounds:
\begin{equation*}
    B\overline{K}\overline{x}+\underline{L}Cx\ge (B\overline{K}+\underline{L}C)x\ge(B\overline{K}+\underline{L}C)\underline{x}.
\end{equation*}
Substituting this in \eqref{eq:lowerpos} results in ${(A+B(\overline{K}+\underline{K}))\underline{x}\ge0}$, which holds on all of $\mathcal{X}_e$ under \eqref{subeq:ABKK}.

(${ii\textnormal{)}\implies\textnormal{(}i}$): As in the above argument for necessity, we examine two cases; first, take ${\underline{e}=\mathbf{0}}$. Necessity of ${A+B(\overline{K}+\underline{K})\ge0}$ in \eqref{eq:errors} is clear when we also have ${\overline{e}=\mathbf{0}}$. Further, if $B\overline{K}$ contains any negative elements, then we can always violate positivity of the dynamics by finding a sufficiently large value of the corresponding element of $\overline{e}$, without leaving $\mathcal{X}_e$. Thus ${B\overline{K}\ge0}$ is necessary. In the second case, let ${\underline{e} = x}$. If ${\overline{e} = \mathbf{0}}$, the closed-loop dynamics of the state become $x^+ = (A+B\overline{K})x$, showing necessity of ${A+B\overline{K}\ge0}$. Finally, consider again the inequality \eqref{eq:lowerpos}, with $\underline{x}=\mathbf{0}$ and $\overline{x}=x$. This shows that the constraint $B\overline{K}+\underline{L}C\ge0$ is also necessary to guarantee invariance of $\mathcal{X}_e$. 
\end{proof}
\begin{remark}
    The final condition \eqref{subeq:BKLC} of Lemma~\ref{lemma:1} is somewhat surprising, as it couples the design of observer and feedback gains. Imposing the additional requirement that ${\underline{L}C\ge0}$ decouples the design problems as a consequence of~\eqref{subeq:BK}. This is used in Theorem~\ref{th:stable} to illustrate that stabilization is possible by combining existing synthesis methods for observer and feedback gains. However, as we show below in Section IV, the relaxed condition \eqref{subeq:BKLC} expands the class of systems that can be stabilized using positive observers beyond previous results.
\end{remark} 
\begin{remark}
    The reader may notice that~\eqref{subeq:ABK} is in fact redundant under the assumption~$A\ge0$. We keep it here nonetheless to facilitate the later discussion of the consequences of lifting the assumption of positivity. 
\end{remark}
With these conditions in place, we are ready to state the first of our main results.
\begin{theorem}\label{th:stable}
    The following statements are equivalent:
    \begin{itemize}
        \item[$(i)$] There exist matrices $K\in\mathbb{R}^{m\times n}$ and $L\in\mathbb{R}^{n\times p}$ such that $A-LC$ and $A+BK$ are Schur, $A-LC\ge0$, $A+BK\ge0$ and $LC\ge0$.
        \item[$(ii)$] The system \eqref{eq:sys} can be stabilized with positive estimates $\overline{x}$ and $\underline{x}$ with $\underline{L}C\ge0$, such that $\mathcal{X}$ {is} invariant under the dynamics \eqref{eq:fullx}.
    \end{itemize}
\end{theorem}
\begin{proof}
    If ($i$) holds, it is straightforward to verify that the choice ${\overline{L}=\underline{L}=L}$, ${\overline{K}=\mathbf{0}}$ and ${\underline{K}=K}$ satisfies the conditions in Lemma \ref{lemma:1} and stabilizes the dynamics, implying~($ii$). For the converse direction, ($ii$) implies that we have ${\rho(A+B(\overline{K}+\underline{K}))<1}$ and, according to Lemma~\ref{lemma:1}, ${A+B(\overline{K}+\underline{K})\ge0}$. Thus, taking ${K=\overline{K}+\underline{K}}$ and $L=\underline{L}$ implies that ($i$) holds. 
\end{proof}
\begin{remark}
    Note that the matrices $K$ and $L$ satisfying the requirements of the first point of Theorem \ref{th:stable} can be found independently and under the conditions for positive control and observation, respectively characterized in \cite{aitrami07control}, \cite{aitrami06observation}. As shown in these works, each problem is equivalent to feasibility of a linear program.
\end{remark}
The following corollary is a direct consequence of \eqref{subeq:BK} in Lemma~\ref{lemma:1}. 
\begin{corollary}\label{co:lower}
    The stabilizing observer of Theorem 1, if it exists, can always be realized with $\overline{K} = \mathbf{0}$. 
\end{corollary}
This alleviates the need for any prior knowledge of the state in order to initialize a stabilizing observer in $\mathcal{X}$, by simply setting $\overline{K}=\mathbf{0}$ and $\underline{x}(0) = \mathbf{0}$. 

\section{Relaxed conditions}

A comparison between the conditions in Lemma~\ref{lemma:1} and the requirements for existing synthesis methods for positive observers~\cite{aitrami06observation} and positivity-preserving linear feedback~\cite{aitrami07control} shows several interesting discrepancies. Structuring the observer as lower and upper estimates and considering the full closed-loop dynamics \eqref{eq:fullx} removes and relaxes some of the constraints required to preserve positivity when treating the problems of positive control and estimation separately. Specifically, for a positive system \eqref{eq:sys} with feedback $u=K\hat{x}$ based on an unstructured positive observer on the form \eqref{eq:obs}, the following (among other conditions) are necessary{\cite{aitrami06observation}}:
\begin{subequations}\label{eq:generic}
    \begin{align}
        A+BK&\ge0\label{subeq:clgain}\\
        BK&\ge0\label{subeq:fbgain}\\
        LC&\ge0.\label{subeq:obsgain}
    \end{align}
\end{subequations}
In the combined formulation described in the previous section, \eqref{subeq:fbgain} holds only for the feedback gain $\overline{K}$ from the upper estimate~$\overline{x}$. It is this relaxation that enables stabilization in Theorem~\ref{th:stable}. Next, we analyze the consequences of relaxing the conditions \eqref{subeq:clgain} and \eqref{subeq:obsgain}. 

Unlike the generic case \eqref{eq:generic}, the observer gain $\overline{L}$ in Lemma~\ref{lemma:1} has no additional restriction apart from Schur stability and nonnegativity of the closed loop error dynamics $A-\overline{L}C$. In the decoupled setting of Theorem~\ref{th:stable} this is of limited use, since Corollary~\ref{co:lower} shows that asymptotic stability of the system depends solely on feedback from the lower estimate $\underline{x}$. In general, however, the observer gain for the lower estimate is subject to the lighter (compared to \eqref{subeq:obsgain}) requirement \eqref{subeq:BKLC}, which relaxes previous conditions for the existence of positive observers.  We illustrate this result with the following simple example.

\begin{example}
    Consider the two-state system 
    \begin{equation}
    \begin{aligned}\label{eq:exsys}
        x(t+1) &= \begin{bmatrix} 1.2 & 0.2 \\ 0 & 0.2 \end{bmatrix}x(t) + u(t)\\
        y(t) &= \begin{bmatrix}
            1 & -1
        \end{bmatrix}x(t).
    \end{aligned}
    \end{equation}
    The only admissible observer gain under \eqref{subeq:obsgain} is, trivially, ${L=\mathbf{0}}$. Regardless of this, the following observer and feedback gains give stabilization using positive estimates:
    \begin{equation}\label{eq:exgains}
        \!\!\overline{L} = \underline{L} = \begin{bmatrix}
            0.3 \\ 0
        \end{bmatrix}\!, \;\; \overline{K} = \begin{bmatrix}
            0 & 0.3 \\ 0 & 0
        \end{bmatrix}\!, \;\; \underline{K} = \begin{bmatrix}
            -0.3 & 0 \\ 0 & 0
        \end{bmatrix}\!.
    \end{equation}
    Insertion shows that the gains in \eqref{eq:exgains} satisfy \eqref{eq:conditions} and give Schur stable closed-loop and error dynamics. The evolution of the first state and corresponding estimates are plotted in Figure \ref{fig:countersim}. This also demonstrates the fact that Corollary~\ref{co:lower} is restricted to the setting of Theorem~\ref{th:stable} and does not hold in general; some systems require feedback from the upper estimate to enable stabilization using positive observers.
\end{example}

\begin{figure}
    \centering
    \begin{tikzpicture}
\begin{axis}[
    width=8cm, height=6cm,
    xmin=0,
    xlabel={Time $t$},
    ylabel={Magnitude},
    grid=both,
    axis lines=left,
    legend style={at={(0.98,0.42)},anchor=south east}
]

\addplot[blue!70!black, thick] table [x index=0, y index=1, col sep=comma] {state_bounds.csv};
\addlegendentry{$x_1$}

\addplot[red!70!black, dashed] table [x index=0, y index=2, col sep=comma] {state_bounds.csv};
\addlegendentry{$\overline{x}_1,\; \underline{x}_1$}

\addplot[red!70!black, dashed] table [x index=0, y index=3, col sep=comma] {state_bounds.csv};

\end{axis}
\end{tikzpicture}
    \caption{Evolution of the first state component~$x_1$ of~\eqref{eq:exsys} and corresponding upper and lower estimates $\overline{x}_1$ and~$\underline{x}_1$ with observer and feedback gains according to~\eqref{eq:exgains}. The component~$x_1$ contains the unstable mode in~\eqref{eq:exsys}, which is stabilized by feedback from the lower estimate when the observer state approaches the true value $x$. Each element of the initial state $x(0)$ is drawn randomly from the interval $\left[0,1\right]$, with estimates initialized to ${\overline{x}(0)=\mathbf{1}}$ and ${\underline{x}(0)=\mathbf{0}}$, dictating the transient performance. The chosen gains stabilize the errors and dynamics.}
    \label{fig:countersim}
\end{figure}
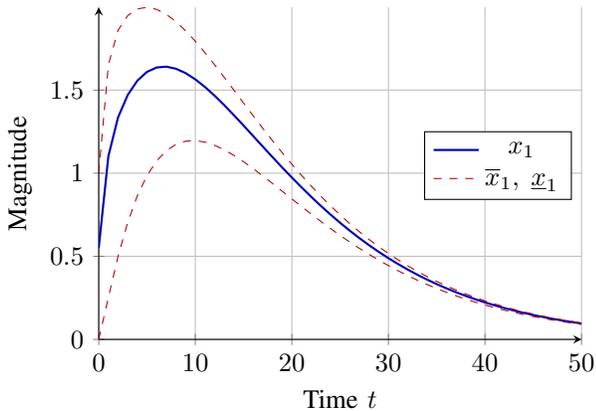

\subsection{Feedback positivization}

The closed-loop gain requirements~\eqref{subeq:ABKK} and~\eqref{subeq:ABK} only apply to the upper feedback $A+B\overline{K}$ and the combined feedback ${A+B(\overline{K}+\underline{K})}$. While we have so far been working under the assumption of positive autonomous dynamics $A\ge0$, this is not required for Lemma~\ref{lemma:1} to hold and can be lifted while maintaining the invariance of $\mathcal{X}$. We can interpret this as the feedback from the upper bound "positivizing" the system, in the sense that the previous condition on $A\ge0$ on the autonomous dynamics is replaced by~\eqref{subeq:ABK}. This is similar to the analysis in \cite{briat11ILC} where, although no observers are employed, feedback from the system output is used to make the closed-loop dynamics positive. 

In applications where we desire to keep the state nonnegative but the system~\eqref{eq:sys} is not inherently positive, Lemma~\ref{lemma:1} provides a means to enforce such behavior as long as the conditions~\eqref{eq:conditions} can be satisfied. A frequent application of the theory of positive systems is the setting of logistics networks, where it may be critical to ensure that no part of the network is at a deficit. Furthermore, methods for optimal control such as those proposed in~\cite{ohlin25heuristic} require positivity of the state to ensure well-posedness of the optimization problem. Positivization of the dynamics using feedback from an upper estimate provides a means to apply these results even when positivity of the original system is not guaranteed. 

\begin{example}
    Modifying the system in Example~1, we instead consider the dynamics
    \begin{equation}
        \begin{aligned}\label{eq:ex2sys}
        x(t+1) &= \begin{bmatrix} 1.2 & 0.2 \\ -0.1 & 0.2 \end{bmatrix}x(t) + u(t)\\
        y(t) &= \begin{bmatrix}
            1 & -1
        \end{bmatrix}x(t).
    \end{aligned}
    \end{equation}
    We no longer have $A\ge0$, but can nevertheless find gains such that Lemma~\ref{lemma:1} applies by making minor alterations to the observer and control design~\eqref{eq:exgains}. Choosing
    \begin{equation}\label{eq:ex2gains}
        \!\!\overline{L} = \underline{L} = \begin{bmatrix}
            0.3 \\ -0.1
        \end{bmatrix}\!, \;\; \overline{K} = \begin{bmatrix}
            0 & 0.3 \\ 0.1 & 0
        \end{bmatrix}\!
    \end{equation}
    and keeping $\underline{K}$ as in~\eqref{eq:exgains}, we see that the new observer gain~$L$ ensures positivity of the error dynamics~\eqref{subeq:ALCu}-\eqref{subeq:ALCl} while maintaining stability. Meanwhile, $\overline{K}$ is increased to ensure positivity in~\eqref{subeq:ABKK}, without destabilizing the closed loop dynamics.
\end{example}

\begin{remark}
    We note the distinct roles of the upper and lower estimates $\overline{x}$ and $\underline{x}$ in the presented construction.  This constitutes a natural dual to the monotone dynamics of upper and lower bounds on the linear copositive Lyapunov function ${V(x)=\lambda^\top x}$ under value iteration \cite{ohlin25heuristic}. 
\end{remark}

\section{Stochastic noise}

The above results rely on unperturbed dynamics and measurements in the system \eqref{eq:sys}. In order to understand the performance of positive observers in a more realistic setting, we examine the system
\begin{equation}\label{eq:noisys}
    \begin{aligned}
        x(t+1) &= Ax(t) + Bu(t) + Ew(t)\\
        y(t) &= Cx(t) + Fv(t)
    \end{aligned}
\end{equation}
where $w$ and $v$ are positive stochastic noises, in keeping the the convention of positive signals. Since we assume that the autonomous dynamics are positive, the process noise is restricted, letting ${E\ge0}$. No restriction is placed on~$F$, allowing for measurement noise in both directions. Further, let the noises be sampled from some fixed distribution with unit first moments ${\mathbb{E}\left[w(t)\right]=\mathbf{1}}$ and ${\mathbb{E}\left[v(t)\right]=\mathbf{1}}$, so that the noise level is modeled completely by the matrices~$E$ and~$F$. Note that the above assumptions on the disturbances are lighter than those of~\cite{briat11ILC}.

Without detailed knowledge of the structure of the noise, it is no longer possible to guarantee invariance of $\mathcal{X}$ at every time step. We can still, however, characterize the expected behavior of the system using our previous analysis for the deterministic case. For this purpose, denote the expected value of the state and estimates as
\begin{equation}\label{eq:moment}
    X(t) = \mathbb{E}\left(\begin{bmatrix}
        x(t) \\ \overline{x}(t) \\ \underline{x}(t)
    \end{bmatrix}\right).
\end{equation}

The following theorem gives conditions for stability and positivity in expectation.
\begin{theorem}\label{th:noise}
    Consider the system~\eqref{eq:noisys} with feedback law ${u(t) = \overline{K}\;\overline{x}(t) + \underline{K}\;\underline{x}(t)}$. Let the upper and lower estimates according to Definition~\ref{def:estimates} be initialized so that ${X(0)\in\mathcal{X}}$. Then, the following are equivalent:
    \begin{itemize}
        \item[($i$)] The {closed-loop dynamics $A + B(\overline{K} + \underline{K})$, $A-\overline{L}C$ and $A-\underline{L}C$ are Schur.} The feedback and observer gains satisfy~\eqref{eq:conditions} and
            \begin{subequations}\label{eq:newconditions}
            \begin{align}
                \overline{L}F\mathbf{1}-E\mathbf{1}&\ge0\label{subeq:LFE}\\
                E\mathbf{1}-\underline{L}F\mathbf{1}&\ge0\label{subeq:ELF}\\
                \underline{L}F\mathbf{1}&\ge0\label{subeq:LF}
            \end{align}
            \end{subequations}
        \item[($ii$)] 
        The expected state $X(t)$ lies invariantly in~$\mathcal{X}$ and converges asymptotically to a fixed point ${X^*\in\mathcal{X}}$ for all initial states~$x(0)\in\mathbb{R}^n_+$. 
    \end{itemize}
\end{theorem}
\begin{proof}
    Define the expected error state 
    \begin{equation}
        X_e = \mathbb{E}\left(\begin{bmatrix}
        x \\ \overline{e} \\ \underline{e}
    \end{bmatrix}\right).
    \end{equation}
    The closed-loop dynamics of $X_e$ are given by
    \begin{equation}\label{eq:xeplus}
        X_e^+ = GX_e + \begin{bmatrix}
            E & \mathbf{0} \\ -E & \overline{L}F \\ E & -\underline{L}F
        \end{bmatrix}\mathbf{1}
    \end{equation}
    where $G$ is the closed-loop matrix of~\eqref{eq:errors}. 
    
    (${i\textnormal{)}\implies\textnormal{(}ii}$): Given ($i$), $G$ is Schur stable, which means that the affine equation~\eqref{eq:xeplus} converges asymptotically to a fixed point. It remains to show invariance of ${X\in\mathcal{X}}$.
    We inspect~\eqref{eq:xeplus}, observing that the left-hand term lies in~$\mathcal{X}_e$ according to Lemma~\ref{lemma:1}. The conditions~\eqref{subeq:LFE} and~\eqref{subeq:ELF} are sufficient to ensure nonnegativity of the right-hand term. Applying~\eqref{subeq:LF} to~\eqref{eq:xeplus} shows that the third block element of the right-hand term is less that or equal to the first. In other words, the right-hand term in~\eqref{eq:xeplus} lies in~$\mathcal{X}_e$. Recall that~$\mathcal{X}_e$ is a cone, so the sum of two members must also lie in~$\mathcal{X}_e$. Since ${X_e\in\mathcal{X}_e\iff X\in\mathcal{X}}$ and $\mathcal{X}$ is closed, it holds that ${X^*\in\mathcal{X}}$.
    
    (${ii\textnormal{)}\implies\textnormal{(}i}$): It is clear from~\eqref{eq:xeplus} that stability of the block diagonal elements of $G$ (see~\eqref{eq:errors}) is necessary. Inserting ${\mathbb{E}[\,\overline{e}\,]=\mathbb{E}[\,\underline{e}\,]=\mathbb{E}[x]=\mathbf{0}}$ into~\eqref{eq:xeplus} directly yields necessity of~\eqref{eq:newconditions}. To see the necessity of~\eqref{eq:conditions}, consider a choice of observer gains that violate~\eqref{eq:conditions}. By Lemma~\ref{lemma:1}, this means that $\mathcal{X}_e$ is not invariant under the transformation $G$. Since the right-hand term of~\eqref{eq:xeplus} is constant and $\mathcal{X}_e$ is a cone, linearity then implies that we can find some ${X_e\in\mathcal{X}_e}$ such that ${X_e^+\notin\mathcal{X}_e}$. This contradicts the statement of~($ii$), leading to the conclusion that~\eqref{eq:conditions} is indeed necessary.
\end{proof}

\begin{remark}
    In the analysis of previous sections there was little to motivate a separate choice of upper and lower observer gains. The lighter requirements on $\overline{L}$ in~\eqref{eq:conditions} allowed for a minor improvement in the convergence rate of the upper bound, but the essential stability and positivity properties of the system were the same as for ${\overline{L}=\underline{L}}$. In the noisy setting~\eqref{eq:noisys} this separation has stronger motivation; indeed, the conditions~\eqref{eq:newconditions} are not satisfied with the choice $\overline{L}=\underline{L}$ for any nonzero noise dynamics $E$ and $F$.
\end{remark}

We showcase the behavior of a noisy system on the form~\eqref{eq:noisys} stabilized using observer feedback according to Theorem~\ref{th:noise} in the following example.

\begin{example}
    Consider the system
    \begin{equation}\label{eq:ex3sys}
    \begin{aligned}
        x(t+1) &= \begin{bmatrix} 0.9 & 0.2 \\ 0.5 & 0.2 \end{bmatrix}x(t) + u(t) + 0.02w(t)\\
        y(t) &= \begin{bmatrix}
            1 & -1
        \end{bmatrix}x(t) + 0.06v(t)
    \end{aligned}
    \end{equation}
    Gains satisfying~\eqref{eq:conditions} and~\eqref{eq:newconditions} are given by
    \begin{equation}
        \begin{aligned}
            \overline{L} &= \begin{bmatrix}
                0.6 \\ 0.5
            \end{bmatrix}, \;\;\;\;\;\;\;\; \underline{L} = \begin{bmatrix}
                0.2 \\ 0.2
            \end{bmatrix}\\
            \overline{K} &= \begin{bmatrix}
                0 & 0.3 \\ 0 & {0.2} 
            \end{bmatrix}, \;\; \underline{K} = \begin{bmatrix}
                -0.3 & 0 \\ 0 & 0 
            \end{bmatrix}.
        \end{aligned}
    \end{equation}
    Figure~2 shows the trajectories of the first component of the state and estimates. The positive noises $w(t)$ and $v(t)$ are independently drawn from a gamma distribution with unit mean and scale parameter.
\end{example}

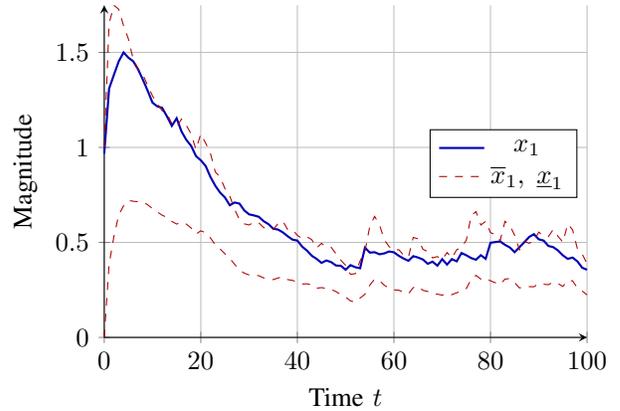
\begin{figure}
    \centering
    \begin{tikzpicture}
\begin{axis}[
    width=8cm, height=6cm,
    xmin=0,
    xlabel={Time $t$},
    ylabel={Magnitude},
    grid=both,
    axis lines=left,
    legend style={at={(0.98,0.42)},anchor=south east}
]

\addplot[blue!70!black, thick] table [x index=0, y index=1, col sep=comma] {state_bounds_noisy.csv};
\addlegendentry{$x_1$}

\addplot[red!70!black, dashed] table [x index=0, y index=2, col sep=comma] {state_bounds_noisy.csv};
\addlegendentry{$\overline{x}_1,\; \underline{x}_1$}

\addplot[red!70!black, dashed] table [x index=0, y index=3, col sep=comma] {state_bounds_noisy.csv};

\end{axis}
\end{tikzpicture}
    \caption{Trajectory of the first state component $x_1$ and corresponding bounds for the system in Example~3. Due to the process and measurement noise, the bounds are no longer guaranteed to hold at every time step $t$. Regardless, invariance ${X\in\mathcal{X}}$ of the expected behavior is sufficient to guarantee asymptotic convergence to a bounded steady state, despite instability of the autonomous dynamics~\eqref{eq:ex3sys}.}
    \label{fig:noisim}
\end{figure}

\section{Conclusion}

We have demonstrated the elementary properties of positive Luenberger-type observers for linear positive systems. Building on earlier results, we have shown that it is possible to stabilize such systems using linear observer feedback without prior knowledge of the state. The proposed approach naturally divides observers into upper and lower estimates as a consequence of the monotonicity of positive dynamics. We further show that, when considering the problems of positive observation and control jointly, our method provides guarantees on stability under feedback from the state estimate for systems not admissible under the constraints of previous observer designs. The introduction of stochastic disturbances yields a second set of linear inequalities restricting the choice of observer gains. This motivates the use of separate observer dynamics to guarantee positive stabilization in the presence of process and measurement noise.

We believe that these results are an important step towards constructing a theory of optimal control and estimation for positive systems. It is our hope that ongoing work on the duality between the construction of positive observers and closed-form solutions to optimal control will yield a framework for positive systems, mirroring the classical results of LQG for control on invariant semidefinite cones.



\bibliographystyle{IEEEtran}
\bibliography{references}

\end{document}